\documentclass[12pt]{amsart}
\usepackage[dvipsnames]{xcolor}
\usepackage{url,graphicx,xypic,tikz,amssymb}
\usepackage{hyperref}
\usepackage[nameinlink,capitalize]{cleveref}
\newcommand\myshade{85}
\colorlet{mylinkcolor}{violet}
\colorlet{mycitecolor}{YellowOrange}
\colorlet{myurlcolor}{Aquamarine}

\hypersetup{
  linkcolor  = mylinkcolor!\myshade!black,
  citecolor  = mycitecolor!\myshade!black,
  urlcolor   = myurlcolor!\myshade!black,
  colorlinks = true,
}

\usepackage[margin=1.5in]{geometry}

\newtheorem{theorem}{Theorem}[section]
\newtheorem*{theorem*}{Theorem}
\newtheorem{proposition}[theorem]{Proposition}

\newtheorem{corollary}[theorem]{Corollary}
\newtheorem*{corollary*}{Corollary}
\newtheorem{conjecture}[theorem]{Conjecture}

\newtheorem{Atheorem}{Theorem}

\theoremstyle{definition}
\newtheorem{example}[theorem]{Example}

\newtheorem{remark}[theorem]{Remark}

\def\AA{\mathbf{A}} % for affine space
\def\CC{\mathbf{C}}
\def\C{\mathcal{C}} % for cohomology classes, a la Miller--Sturmfels

\def\K{\mathcal{K}} % for K-polynomial
\def\LL{\mathcal{L}}
\def\NN{\mathbf{N}}
\def\OO{\mathcal{O}} 
\def\PP{\mathbf{P}}

\def\QQ{\mathbf{Q}}

\def\SS{\mathbf{S}}
\def\S{\mathcal{S}}
\def\ZZ{\mathbf{Z}}

\def\GL{\mathrm{GL}}

\def\Sym{\mathrm{Sym}}
\def\Hilb{\mathrm{Hilb}}
\def\codim{\mathrm{codim}}
\def\Spec{\operatorname{Spec}}

\title{Equivariant $K$-theory classes of matrix orbit closures}
\author{Andrew Berget}
\address{Western Washington University, Bellingham, WA, USA}
\email{andrew.berget@wwu.edu}
\author{Alex Fink}
\address{Queen Mary University of London, London, UK}
\email{a.fink@qmul.ac.uk}

\begin{document}
\maketitle

\begin{abstract}
  The group $G = \GL_r(k) \times (k^\times)^n$ acts on
  $\AA^{r \times n}$, the space of $r$-by-$n$ matrices: $\GL_r(k)$
  acts by row operations and $(k^\times)^n$ scales columns. A matrix
  orbit closure is the Zariski closure of a point orbit for this
  action. We prove that the class of such an orbit closure in $G$-equivariant $K$-theory of $\AA^{r \times n}$ is determined by
  the matroid of a generic point. 
  We present two formulas for this class.  The key to the proof
  is to show that matrix orbit closures have rational singularities.
\end{abstract}
\section{Introduction}
Let $r$ and $n$ be integers, $r \leq n$, and $\AA^{r \times n}$ the
affine space of $r$-by-$n$ matrices with entries in an algebraically
closed field $k$ of characteristic zero. We consider the left action
of $\GL_r = \GL_r(k)$ on $\AA^{r \times n}$ by row operations, and the
right action of $T^n = (k^\times)^n$ by scaling columns. Let
$v \in \AA^{r \times n}$ be a matrix, and consider
$X_v^\circ = \GL_r v T^n$, which is the orbit of $\GL_r \times T^n$
through $v$. We call the Zariski closure $X_v = \overline{X_v^\circ}$
a \emph{matrix orbit closure}, and it is our primary object of
interest. Matrix orbit closures were studied in \cite{chow,moc} and
generalizations of them were studied in \cite{binglin,spinkTseng}. A
primary focus of these papers is how knowledge of which collections of
columns of $v$ form linearly independent sets, data known as the
matroid of $v$, affects both the geometry and the algebraic invariants
of the matrix orbit closure.

Write $G = \GL_r \times T^n$. We consider the Grothendieck group of
$G$-equivariant coherent sheaves on $\AA^{r \times n}$, denoted
$K^G_0(\AA^{r \times n})$. Since $\AA^{r \times n}$ is an affine
space, this group can be identified with the representation ring of
$G$. As such, the class of a coherent sheaf on $\AA^{r \times n}$ can
be written as a Laurent polynomial with integer coefficients in
variables $u_1,\dots,u_r$, $t_1,\dots,t_n$, which generate this
representation ring (see \cref{ex:tu}).  We view the class of $X_v$ in
$K^G_0(\AA^{r \times n})$ as a proxy for how complicated $X_v$ is.
Essentially the same class is also studied in the guise of the multigraded
Hilbert series of $X_v$; either invariant is readily extracted from
the other.

Our main goal is to prove the following result.
\begin{Atheorem}\label{thm:a}   Let $v \in \AA^{r \times n}$ be any matrix. 
  \begin{enumerate}
  \item The class of $X_v$ in
  $K^G_0(\AA^{r \times n})$ can be determined from the matroid of
  $v$ alone.
\item Assume that $v$ is a rank $r$ matrix, and denote its matroid by~$M$. 
  Then, the sum
  \[
    \K(M) = \sum_{w \in S_n} \prod_{j \notin B(w)}\prod_{i \in [r]} (1-
    u_it_j) \cdot \prod_{i=1}^{n-1} \frac{1}{1-t_{w_{i+1}}/t_{w_{i}}},
  \]
  \emph{a priori} a rational function,
  is a polynomial in $u_1,\dots,u_r$ and $t_1,\dots,t_n$, and it represents the class of $X_v$ in $K^G_0(\AA^{r \times n})$. Here
  $[r] = \{1,2,\dots,r\}$, $S_n$ is the symmetric group on $[n]$, and
  for $w = (w_1,\dots,w_n) \in S_n$, $B(w)$ is the lexicographically
  first basis of $M$ in the list $w$.
  \end{enumerate}
\end{Atheorem}
There are two main motivators for this result. The first comes from
\cite{speyer,finkSpeyer}, where for any matroid $M$ of rank $r$ on $n$
elements a class $y^{T^n}(M)$ is defined in the $T^n$-equivariant $K$-theory
of the Grassmannian $Gr(r,n)$. This class is defined piecewise, using
equivariant localization; \cref{thm:a} provides an explicit lift of
$y^{T^n}(M)$ to a single polynomial expression. A similar result on
equivariant Chow classes was recently obtained by Lee, Patel, Spink
and Tseng \cite[Theorem 1]{spinkTseng}. In \cref{prop:chow} we derive
\cite[Theorem 1]{spinkTseng} from \cref{thm:a}.

The second motivation for \cref{thm:a} comes from the following problem:
Given a matroid $M$, when is it possible to partition the ground set
of $M$ into independent sets of prescribed sizes? When $M$ is realized
by $v \in \AA^{r \times n}$ the answer is contained in the Schur
polynomial expansion of the coefficient of $t_1t_2\dots t_n$ in the
multigraded Hilbert series of $X_v$ (see \cite[Section~8.2]{moc}). The
anecdotal matroid invariance of this Schur polynomial expansion
motivated the authors to connect the tensor modules in
\cref{sec:consequences} to equivariant $K$-classes and ultimately
conjecture the matroid invariance of the equivariant $K$-class of
$X_v$ in \cite[Conjecture 5.1]{moc}.

The key to proving \cref{thm:a} is the following result.
\begin{Atheorem}\label{thm:b}
  Let $v \in \AA^{r \times n}$ be a rank $r$ matrix. Then $X_v$ has
  rational singularities.
\end{Atheorem}
This means, roughly, that the cohomological behavior of the structure
sheaf of $X_v$ does not change on desingularization. As a consequence of
\cref{thm:b} we show the following result.
\begin{Atheorem}\label{thm:c}
  Let $Y$ be a $T$ orbit closure in $Gr(r,n)$. Let $\mathcal{S}$ be
  the tautological bundle on $Gr(r,n)$ and $\mathbf{S}^\lambda$ a
  Schur functor where $\lambda$ is a partition with at most $r$
  parts. Then, for all $m \geq 1$,
  \[
    H^m(Y,\mathbf{S}^\lambda(\mathcal{S}^*)) = 0.
  \]
\end{Atheorem}
We view this as a variant of the Borel-Weil-Bott theorem for torus
orbit closures in $Gr(r,n)$. 

There are significant parallels between Theorems \ref{thm:a},
\ref{thm:b} and \ref{thm:c} and results occurring in Schubert
calculus. The structure of the equivariant $K$-theory and Chow ring of
the Grassmannian is governed by the geometry of the Schubert
varieties, which are orbit closures of a Borel subgroup
$B \subset \GL_n$. Taking $\GL_r \times B$ orbit closures in
$\AA^{r \times n}$ yields the \textit{matrix Schubert varieties}. The
equivariant $K$-classes of matrix Schubert varieties provide canonical
representatives for the equivariant $K$-classes of Schubert
varieties. Happily, the choice furnished by the matrix analogue
coincides exactly with the well-studied (double) Grothendieck
polynomials \cite[Theorem A]{km}. Our \cref{thm:a} presents a similar
result for torus orbit closures $Y$ in Grassmannians, where now the
polynomial representing the $K$-class depends only on the matroid
stratum of a point in the big orbit of $Y$. 

The methods of \cite{km} are those of Gr\"obner degenerations and
avoid rational singularities entirely. In our case $X_v$ does have a
nice Gr\"obner degeneration \cite[Section 4]{binglin}, but a key
feature of this degeneration is that it purposefully 
breaks the symmetry under $\GL_r \subset G$, 
making it ill-suited for our purposes. 
An interesting open problem is to prove \cref{thm:a} via the
combinatorics of Gr\"obner degenerations. An important step in this
process was completed in \cite{spinkTseng} where the identification of
the $G$-equivariant Chow class of $X_v$ (which ignores phenomena of
positive codimension within $X_v$) was completed using degenerative
techniques. Thus, while our \cref{thm:a} subsumes the statement of
\cite[Theorem~1]{spinkTseng}, their independent proof may ultimately
be of considerable value.

Given a $T^n$-invariant subvariety $Y$ of $Gr(r,n)$, one may define
the matrix analogue of $Y$ to be the closure in $\AA^{r \times n}$ of
those matrices whose row span is a point of $Y$. Generally, one should
not expect to be able to say anything non-trivial about the
equivariant $K$-class of the matrix analogue of $Y$ without some
additional hypothesis. Rational singularities of $Y$ and its matrix
analogue appears to be the right condition in the study of matrix
orbit closures.

Matrix Schubert varieties are shown to have rational singularities in
\cite[Theorem 2]{kempf}, and this result is based on the usual
Schubert varieties having rational singularities, a result shown using
Frobenius splittings in \cite[Theorem 4]{ramanathan} (see also
\cite{brion2,brion}).

In our \cref{thm:b} we deduce rational singularities of $X_v$ by
studying the quotient of $X_v$ by the $n$-torus $T^n$, which is a
subvariety of $(\PP^{r-1})^n$. We show that this quotient has rational
singularities by applying results of Brion \cite{brion2, brion} and Li
\cite{binglin} on multiplicity-free varieties. An interesting open
question is whether $X_v$ is compatibly Frobenius split in
$\AA^{r \times n}$. The Borel-Weil-Bott style result of \cref{thm:c}
is a close relative of \cref{thm:b}. Its analogue in Schubert calculus
was studied extensively in \cite[Theorem 1]{kempf2}. A problem
for future study is the generalization \cref{thm:c} to Coxeter
matroids in flag varieties of other types.  Recently, \cite[Section
10]{best} initiated the study of a systematic procedure to equate
certain equivariant $K$-theoretic Euler characteristics with
equivariant integrals in Chow theory. We wonder whether combining
\cref{thm:c} with this method provides any insight in the study of
$H^0(Y,\mathbf{S}^\lambda(\mathcal{S}^*))$ (where $Y$ is a torus orbit
closure in $Gr(r,n)$).

After proving Theorems~\ref{thm:a}, \ref{thm:b} and \ref{thm:c} we
investigate $K$-theoretic positivity in the sense of \cite{agm}. The
main result here is \cref{prop:pos1}, which says that, when working
over $\CC$, the equivariant $K$-class of $X_v$ expands ``positively''
in terms of double Grothendieck polynomials. Once again, this result
hinges on the rational singularities of $X_v$. We use this result as
motivation for a series of progressively weaker conjectures on
positivity properties of both $K$-classes and Chow classes of
$X_v$. It is our hope that these conjectures provide fertile ground
for future work in this area.

\subsubsection*{Acknowledgements} The authors would like to thank Dave
Anderson and David Speyer for reading early drafts. Thanks are also due to
Hunter Spink and Dennis Tseng for useful conversations, as well as the
referees for their helpful comments.

\tableofcontents
\section{Background on equivariant $K$-theory}\label{sec:K}
Let $k$ be an algebraically closed field. A variety will be an integral scheme of
finite type over $k$. Let $X$ be a variety with a $G$-action, where
$G$ is a linear algebraic group. General references for the material
discussed below are \cite[Chapter 5]{ginzburg} and \cite{handbook}.

We let $K^G_0(X)$ denote the Grothendieck group of $G$-equivariant
coherent sheaves over $X$. We let $K^0_G(X)$ denote the Grothendieck
group of $G$-equivariant vector bundles over $X$. There is a natural
group homomorphism, $K^0_G(X) \to K^G_0(X)$.  Using the tensor product
of vector bundles, $K^0_G(X)$ is a ring, and $K^G_0(X)$ is a module
over $K^0_G(X)$. If $X$ is smooth then this map
$K^0_G(X) \to K_0^G(X)$ is an isomorphism, as every equivariant
coherent sheaf can be resolved by equivariant vector bundles.

If $f:Z \to X$ is a $G$-equivariant proper map then there is a
pushforward $f_* : K^G_0(Z) \to K^G_0(X)$ defined by
$f_* [\mathcal{F}] = \sum_i (-1)^i [R^i f_* \mathcal{F}]$.

Let $R(G)$ denote the representation ring of $G$. Then $K^G_0(X)$ is a
module over $R(G)$. If $X$ is an affine space we can identify
$K^G_0(X) = K^0_G(X) = R(G)$.

It is important to emphasize that in our work, $G$ will always be a
general linear group, a torus, or a product thereof. In this case
$R(G)$ is easy to describe. Let $T$ be a maximal torus of $G$ and let
$W$ be the Weyl group of $G$. The representation ring of $T$ is a
Laurent polynomial ring $\ZZ[\operatorname{Hom}(T,k^\times)]$. Then
$R(G)$ is the ring of $W$-invariants of $R(T)$.

Let $A = k[x_1,\dots,x_m]$ and let $\AA = \operatorname{Spec}(A)$ be an
affine space carrying a $G$ action. An equivariant coherent sheaf on
$\AA$ is described by an equivariant coherent $A$-module. We describe
how to compute the class of such a sheaf in
$K^G_0(\AA) = R(G) \subset R(T)$. 
Let $T$ be the maximal torus of~$G$ and $d$ its dimension, so that
$R(T) = \ZZ[t_1^{\pm 1},\dots,t_d^{\pm 1}]$. The ring
$k[x_1,\dots,x_m]$ is graded by the character group of $T$, which is
$\operatorname{Hom}(T,k^\times) = \ZZ^d$. We assume that the grading
is positive, in that the degrees of the variables $x_i$ lie in a common
open half-space of $\QQ^d\supset\ZZ^d$.

Let $M$ be a finitely generated, $G$-equivariant module over
$k[x_1,\dots,x_m]$. Then $M$ is a $\ZZ^d$-multigraded
$k[x_1,\dots,x_m]$-module. The multigraded Hilbert series of $M$ is
\[
  \Hilb(M) = \sum_{\mathbf{a} \in \ZZ^d} \dim_k(M_{\mathbf{a}})\mspace{2mu}
  t_1^{{a}_1}t_2^{{a}_2} \dots t_d^{{a}_d} \in
  \ZZ[[t_1^{\pm 1},\dots,t_d^{ \pm 1}]].
\]
Using an equivariant resolution, one can write the Hilbert series as a
rational function $\K(M)/\prod_{i=1}^m (1-t^{\deg(x_i)})$, as
explained in \cite[Chapter 8]{millerSturmfels}, where
$\K(M) \in R(G) \subset R(T)$ is referred to as the $K$-polynomial of
$M$. The $K$-polynomial $\K(M)$ represents the class of the sheaf
associated to~$M$ in $K^G_0(\AA)$. When $M$ is the coordinate ring of a closed 
subvariety $Z \subset \AA$, we abuse notation and write $[Z]$ or $\K(Z)$ 
for the class of the sheaf associated to~$M$ in $K^G_0(\AA)$.

\begin{example}\label{ex:tu}
  Let $G = \GL_r \times T^n$ and let $X = \AA^{r \times n}$. Then, the
  maximal torus of $G$ is $T^r \times T^n$, where $T^r$ is the
  diagonal maximal torus in $\GL_r$. We have
  \[
   R(T^r \times T^n) = \ZZ[u_1^{\pm 1},\dots,u_r^{\pm 1},t_1^{\pm 1},\dots,t_n^{\pm 1}],
 \]
 and,
\[
R(G) = \ZZ[u_1^{\pm 1},\dots,u_r^{\pm 1},t_1^{\pm 1},\dots,t_n^{\pm 1}]^{S_r},
\]
where $S_r$ is the symmetric group on $[r]$, which permutes the $u$
variables.

Writing $A = k[x_{ij}: 1 \leq i \leq r, 1 \leq j \leq n]$ and
$\AA^{r\times n}$ for $\operatorname{Spec}(A)$, we have that
$\AA^{r \times n}$ has a $G$-action as described in the
introduction.  
We take the sign convention that the character of the $(T^r \times T^n)$-action on the
one-dimensional $k$-vector space spanned by $x_{ij}$ is $u_it_j$.
Now, if $M$ is a $(T^r \times T^n)$-equivariant $A$-module we can write
\[
  \Hilb(M) = \K(M)\prod_{j \in [n]}\prod_{i \in [r]} \frac{1}{1-u_it_j}.
\] 
In the future we will write $T$ for $T^n$ unless confusion may arise.
\end{example}
\section{Background on rational singularities}\label{sec:ratSing}
Recall that we work over an algebraically closed field of
characteristic zero. A proper birational morphism $f: Z \to Y$ of
varieties, where $Z$ is smooth, is called a resolution of
singularities.  It is called a rational resolution of singularities if
\begin{enumerate}
\item[(i)] $Y$ is normal, i.e., $\mathcal{O}_Y \to f_* \mathcal{O}_Z$
  is an isomorphism, and
\item[(ii)] $R^m f_* \mathcal{O}_Z = 0$ for $m >0$.
\end{enumerate}
We say that $Y$ has rational singularities if there exists a rational
resolution of singularities $f: Z \to Y$. We refer the reader to
Koll\'ar and Mori \cite[Section 5.1]{kollar} for more on these
singularities. The following well known results will be needed.
\begin{proposition}\label{prop:every}
  If one resolution of singularities of $Y$ is rational, then every
  resolution is.
\end{proposition}
%\begin{proposition}\label{prop:partial}
%  Let $f:Z \to Y$ be a proper birational morphism. Assume that $Z$ has
%  rational singularities. If $\mathcal{O}_Y \to f_* \OO_Z$ is an
%  isomorphism and $R^m f_* \OO_Z = 0$ for $m >0$ then $Y$ has rational
%  singularities.
%\end{proposition}

\begin{proposition}\label{prop:Gss}
  Let $f:Z\to Y$ be a proper birational morphism,
  where both $Z$ and $Y$ have rational singularities.
  Then $f_*\OO_Z = \OO_Y$ and $R^mf_*\OO_Z = 0$ for $m>0$.
\end{proposition}
%\begin{proof}
%  Choose a resolution of singularities $g:Z' \to Z$ and compose with
%  $f$ to obtain a resolution of singularities $Z' \to Z \to Y$. Now
%  apply the Grothendieck spectral sequence for $R^m (gf)_* \mathcal{O}_{Z'}$,
%  together with \cref{prop:every}.
%\end{proof}

% \begin{proposition}\label{prop:vb}
%   Let $Z$ be the total space of a vector bundle over $Y$. If $Y$ has
%   rational singularities then $Z$ does too.
% \end{proposition}
% \begin{proof}
%   The projection $Z \to Y$ is flat. The base $Y$ and the fibers (being
%   affine spaces) have rational singularities. By a result of Elkik
%   \cite[{Th\'eor\`em~V}]{elkik}, $Z$ has rational singularities.
% \end{proof}

\begin{proposition}\label{prop:vb}
  Assume that $Z \to Y$ is a locally trivial fiber bundle with smooth
  fiber and normal base. Then $Z$ has rational singularities if and only
  if $Y$ does.
\end{proposition}
By locally trivial fiber bundle we mean the following: There is an
open cover $\{U_i\}$ of $Y$ and a smooth variety $P$ so that for
$g: Z \to Y$ we have $g^{-1}(U_i) = U_i \times P$.
\begin{proof}
  The question is local on $Y$ so we may assume that $Z = Y \times P$
  where $P$ is smooth. The map $P \to \Spec k$ is faithfully flat, and
  hence the base change $g: Y \times P \to Y$ is faithfully flat
  too. Let $f:\widetilde{Y} \to Y$ be a resolution of
  singularities. We also have natural maps
  $g' : \widetilde{Y} \times P \to \widetilde{Y}$ and
  $f' : \widetilde{Y} \times P \to Y \times P$ and these fit together
  into a base change diagram,
  \[
    \xymatrix{
      \widetilde{Y} \times P \ar[r]^{g'}\ar[d]^{f'} & \widetilde{Y} \ar[d]^f \\
      Y \times P \ar[r]^g & Y\rlap.
      }
    \]
    Note that $f'$ is  a resolution of singularities, since $P$ is
    assumed smooth. By flat base change,
    \[
      g^* R^i f_* \mathcal{O}_{\widetilde{Y}} = R^i f'_*
      {g'}^*\mathcal{O}_{\widetilde{Y}} = R^i f'_*
      \mathcal{O}_{\widetilde{Y} \times P}.
    \]
    If $Y$ has rational singularities then the left side above is zero
    for $i> 0$ and hence
    $R^i f'_* \mathcal{O}_{\widetilde{Y} \times P}$ for $i>0$. When
    $i = 0$ then we obtain
    $g^* \mathcal{O}_Y = \mathcal{O}_{Y \times P}$.

    If $Z = Y \times P$ has rational singularities then
    $R^i f'_* \mathcal{O}_{\widetilde{Y} \times P} = 0$ for $i>0$ and
    hence $g^* R^i f_* \mathcal{O}_{\widetilde{Y}} =0$. Because $g^*$
    is faithful, $R^i f_* \mathcal{O}_{\widetilde{Y}} =0$ for
    $i>0$. Since $Y$ is assumed normal we have
    $f_* \mathcal{O}_{\widetilde{Y}} = \mathcal{O}_Y$, so we are done.
\end{proof}

\section{Rational singularities of matrix orbit closures}\label{sec:ratsing}
We now begin our study of matrix orbit closures.
Matroids enter the story in this section.
For those unfamiliar with matroids,
a reference with viewpoint similar to our own is \cite{katzMatroid};
a general text such as \cite{oxley} may also be helpful.

The goal of this section is to prove the following theorem.
\begin{theorem}\label{thm:ratsing}
  Let $v \in \AA^{r \times n}$ be a rank $r$ matrix. Then $X_v$ has
  rational singularities.
\end{theorem}

To prove this, we immediately reduce to the case that $v$ has no zero
columns. Let $p : \AA^{r \times n} \dashrightarrow (\PP^{r-1})^n$ denote
the natural rational map. Then $V=p(X_v)$ is the $\GL_r$ orbit
closure of $p(v)$ in $(\PP^{r-1})^n$. The class of $V$ in the Chow
ring of $(\PP^{r-1})^n$ can be described using a special case of work
of Li \cite[Theorem 1.1]{binglin}. To describe Li's result, recall
that the Chow ring of $(\PP^{r-1})^n$ is isomorphic to
$\ZZ[t_1,\dots,t_n]/(t_1^r,\dots,t_n^r)$. Here, the class of $t_i$
represents the class of a hyperplane in the $i$th factor.

%Recall also that the independence polytope of a matroid $M$ is the convex hull in
%$\mathbf{R}^n$ of the indicator vectors of independent sets of $M$.
For a matroid $M$ with rank function $\operatorname{rk}_M$, define the set
\[S(M) = \{s \in \mathbf{N}^n : \sum_{i\in I}s_i<r\operatorname{rk}_M(I)
\text{ for all $I\subseteq[n]$}, 
\sum_{i=1}^n s_i = r^2 - 1 \}.\]
%The points of $S(M)$ are the maximal lattice points in a certain half-open version
%of the $r$\/th dilate of the independence polytope of~$M$.
The elements of $S(M)$ are all of the lattice points of the Minkowski difference 
$rP(M)-\operatorname{conv}\{e_1,\ldots,e_n\}$,
where $P(M)$ is the basis polytope of~$M$
\cite[Theorem 5.3]{spinkTseng}.

\begin{theorem}[{\cite[Theorem 1.1]{binglin}}]\label{thm:binglin}
  Let $M$ denote the matroid of $v$.
  The class of $V$ in the Chow ring $A^*( (\PP^{r-1})^n )$
  is
  \[
    \sum_{s \in S(M)} \prod_{i=1}^n t_i^{r-1-s_i}.
  \]
  % Then n(r-1) - (r^2 -1 ) = nr - n - r^2  + 1 = nr - (n+r^2-1)
\end{theorem}

The variety $(\PP^{r-1})^n$ is a flag variety, and as such its Chow
ring has a privileged generating set as a $\mathbf Z$-module consisting of the
classes of Schubert varieties.  In the case of $(\PP^{r-1})^n$ the
Schubert varieties are products of linear subspaces in each
$\PP^{r-1}$, whose Chow classes are exactly the monomials in
$t_1,\ldots,t_n$.  It follows from \cref{thm:binglin} that when the
class of $V$ is expressed in this Schubert basis, the coefficients
involved are either $0$ or $1$. Thus, $V$ is \textit{multiplicity
  free} in the sense of Brion \cite{brion}, whose main theorem on such
varieties is this.
\begin{theorem}[{\cite[Theorem~1]{brion}}]\label{thm:brion}
  Let $\mathcal{F}\ell$ be a flag variety of a semisimple algebraic
  group. Let $V \subset \mathcal{F}\ell$ be a subvariety whose Chow
  class is a linear combination of classes of Schubert varieties,
  where all coefficients involved are $0$ or $1$. Then,
  \begin{enumerate}
  \item $V$ is arithmetically normal and Cohen-Macaulay in the
    projective embedding given by an ample line bundle on
    $\mathcal{F}\ell$;
  \item For any globally generated line bundle $\LL$ on
    $\mathcal{F}\ell$, the restriction map
    $H^0(\mathcal{F}\ell, \LL) \to H^0(V,\LL)$ is surjective. All
    higher cohomology groups $H^m(V,\LL)$, $m \geq 1$, vanish. If $\LL$ is ample then $H^m(V,\LL^{-1}) = 0$ for $m < \dim(V)$;\label{brion2}
  \item $V$ has rational singularities.
  \end{enumerate}
\end{theorem}
\begin{proof}
  Only the third item is not part of the statement of
  \cite[Theorem~0.1]{brion}.

  For this item write $\mathcal{F}\ell = G/Q$, where $Q$ is a
  parabolic subgroup. One constructs the variety
  \[
    Y = \{ g \in G : g^{-1} Q/Q \in V\}.
  \]
  By the argument of \cite[Theorem~5]{brion2} one deduces that $Y$ has
  rational singularities. One must assume that $Y$ does not contain a
  $G$ orbit to apply this argument (\textit{cf.}\/\
  \cite[Remark~3.3]{brion} where a small error is noted), which we may
  because if $Y$ did contain a $G$ orbit then $Y = G$, which is smooth
  and thus has rational singularities. Since $Y$ has rational
  singularities and the natural map $Y \to V$ is a locally trivial
  fiber bundle with smooth connected fiber $Q$ and normal base $V$
  (see \cite[Lemma~1.1]{brion}), we conclude by \cref{prop:vb} that
  $V$ has rational singularities.
\end{proof}
Consider now the line bundle $\LL_i$ on $(\PP^{r-1})^n$ whose fiber
over $(\ell_1,\dots,\ell_n)$ is the line $\ell_i \subset \AA^r$. Note that
$\LL_i^{-1}$ is globally generated for all $i$ and the line bundle
$\LL_1^{-1-m_1} \otimes \dots \otimes \LL_n^{-1-m_n}$ is ample over
$(\PP^{r-1})^n$, provided that each $m_i$ is non-negative. We construct the  vector bundle
\[
\mathcal{E} = \LL_1 \oplus  \dots \oplus \LL_n,
\]
which is a subbundle of the trivial bundle with fiber
$(\AA^r)^n = \AA^{r \times n}$. Let
$A = k[x_{ij} : 1 \leq i \leq r, 1 \leq j \leq n]$ be the coordinate
ring of $\AA^{r \times n}$. We can identify
$H^0((\PP^{r-1})^n, \Sym(\mathcal{E}^*))$ with $A$.
\begin{proposition}\label{prop:vanish}
For $V=p(X_v)$, the natural map
  $A \to H^0(V,
  \Sym(\mathcal{E}^*))$ is surjective.
  %For all $m \geq 1$,
  %$H^m(V, \Sym(\mathcal{E}^*)) = 0$.
\end{proposition}
\begin{proof}
  Decompose $\Sym(\mathcal{E}^*)$ as
  $\bigoplus_{(a_1,\dots,a_n) \in \NN^n} \LL_1^{-a_1} \otimes \dots
  \otimes \LL_n^{-a_n}$.  Since
  $\LL^{-a_1} \otimes \dots \otimes \LL^{-a_n}$ is globally generated
  on $(\PP^{r-1})^n$ we may apply \cref{thm:brion}(2), and the result
  follows.
\end{proof}

\begin{proposition}\label{prop:coordring}
  Let $k[X_v]$ denote the coordinate ring of $X_v$. Then there is a
  $G$-equivariant isomorphism of $A$-modules,
  $H^0(V,\Sym(\mathcal{E}^*)) \approx k[X_v]$.
\end{proposition}
\begin{proof}
  Recall that $A$ is multigraded by
  $\operatorname{Hom}(T,k^\times) = \ZZ^r \oplus \ZZ^n$, where the
  degree of $x_{ij}$ is $(e_i,e_j)$. The prime ideal of $X_v$ is
  homogeneous for this grading, i.e.\ is generated by homogeneous elements,
  since $X_v$ is $G$- and hence $T$-invariant.

  The identification $A = H^0((\PP^{r-1})^n, \Sym(\mathcal{E}^*))$ is
  $G$-equivariant, and the restriction map
  $A \to H^0(V, \Sym(\mathcal{E}^*))$ is too. 
  The kernel of the latter is generated
  by those homogeneous polynomials in $A$ whose restriction to the
  orbit $G \cdot v$ is zero; thus the kernel is the prime ideal of~$X_v$,
  which we have just seen is homogeneous. 
  By \cref{prop:vanish} we obtain the desired result.
\end{proof}

Our proof of rational singularities now follows quite quickly from the following
result of Kempf and Ramanathan.  
\begin{theorem}[Kempf, Ramanathan {\cite[Theorem 1]{kempf}}]\label{thm:kempf} Let
  $\LL_1, \dots, \LL_n$ be globally generated line bundles on a
  complete variety $X$ with rational singularities. If
  \begin{enumerate}
  \item $H^i(X,\LL_1^{m_1} \otimes \dots \otimes \LL_n^{m_n})=0$ for all
    $i>0$ and all $m_1,\dots,m_n \in \mathbf{N}^n$ and
  \item $H^i(X,\LL_1^{-1-m_1} \otimes \dots \otimes \LL_n^{-1-m_n}) = 0$
    for all $i < \dim(X)$ and $m_1,\dots,m_n \in \mathbf{N}^n$,
  \end{enumerate}
  then the spectrum of the ring of sections
  $\bigoplus_{(m_1,\dots,m_n) \in \NN^n} H^0(X,\LL_1^{m_1} \otimes
  \dots \otimes \LL_n^{m_n})$ has rational singularities.
  \end{theorem}
\begin{proof}[Proof of \cref{thm:ratsing}]
  \cref{thm:kempf} applies to $X_v$ by \cref{thm:brion} and
  \cref{prop:coordring}.
\end{proof}

\section{Borel-Weil-Bott theorem for torus orbits in Grassmannians}
In this section we use a variant of the Gel$'$fand-Macpherson
correspondence to obtain a variant of the Borel-Weil-Bott theorem for
$T$ orbit closures in $Gr(r,n)$.

The group $G = \GL_r \times T$ acts on $Gr(r,n) = \GL_n/\!/P$, where
$T = T^n$ is the maximal torus of $\GL_n$ and $\GL_r$ acts
trivially. Thus, a $T$ orbit closure in $Gr(r,n)$ is the same thing
as a $G$ orbit closure.

Let $\mathcal{S}$ be the rank $r$ tautological bundle over the
Grassmannian $Gr(r,n)$. Its fiber over a subspace is precisely that
subspace. Recall that the Borel-Weil-Bott theorem for the Grassmannian
describes the cohomology groups of various Schur functors applied to
$\mathcal{S}^*$. It says, in a weakened form, that the higher
cohomology groups of such bundles vanish and gives a formula for their
global sections. Our variant of this result is below.
\begin{theorem}\label{thm:bwb}
  Let $Y$ be a $T$ orbit closure in $Gr(r,n)$. Let $\mathcal{S}$ be
  the tautological bundle over $Gr(r,n)$, and $\SS^\lambda$ be a Schur
  functor where $\lambda$ is a partition with at most $r$ parts. Then
  for all $m \geq 1$,
  \[
    H^m(Y,\SS^\lambda(\mathcal{S}^*)) = 0.
  \]
\end{theorem}
% We delay the description of the global sections of this bundle.

The result will follow by applying Weyman's geometric method
\cite[Chapter 5]{weyman} (\textit{cf.} \cite[p.\ 355, Condition
I$'$]{kempf}), knowing in advance that $X_v$ has rational
singularities.

In this section we assume that $v$ has rank $r$. This means its row
span represents a point in the Grassmannian $Gr(r,n)$.

Let $Y$ be the torus orbit closure through the row space of $v$; $Y$ is
a normal toric variety. Let ${Z}$ be the total space of the vector
bundle $\S^{\oplus r}$ restricted to $Y$, which is a subbundle of the
trivial bundle with fiber $(\AA^n )^r = \AA^{r \times n}$. We put a
$\GL_r$ action on $\S^{\oplus r}$ using the left action on fibers;
since this commutes with the natural $T$ action we see that
$\S^{\oplus r}$ is a $G$-equivariant vector bundle over $Y$. We have a
commutative diagram
\[
  \xymatrix{
    {Z} \ar[r]^{i}\ar[d]^{s'} & \AA^{r \times n} \times Y \ar[d]^s\ar[r]^{\pi} & Y\\
    X_v \ar[r]^{i'} & \AA^{r \times n}
  }
\]
Here the vertical arrows are projections to the first factor, $\pi$ is
projection to the second factor and other maps are inclusions. All
these maps are $G$-equivariant.
\begin{proposition} \label{prop:ratsingGrassmannian}The higher direct
  images $R^m s_* \OO_{{Z}}$, $m \geq 1$, vanish and $s_* \OO_{{Z}}$
  is isomorphic to $k[X_v]$ as a $G$-equivariant $A$-module.
\end{proposition}
\begin{proof}
  The map $s':{Z} \to X_v$ is a partial desingularization, in that it
  is proper and a birational isomorphism: The inverse is the map
  $u \mapsto (u, \textup{rowSpan}(u))$, defined over the set of full
  rank matrices in $X_v$. Since $Y$ is a toric variety it has rational
  singularities and thus ${Z}$ has rational singularities too, by \cref{prop:vb}. 
  It follows from \cref{prop:Gss} that $R^m s'_* \OO_{{Z}} = 0$ for $m \geq 1$ and
  that $s'_* \OO_{{Z}} = \OO_{X_v}$.

  The result follows by viewing these as statements about
  $G$-equivariant $A$-modules (i.e., applying $i'_*$).
\end{proof}
 
\begin{proof}[Proof of \cref{thm:bwb}]
  By \cite[Theorem 5.1.2(b)]{weyman} we may identify
  $R^m s_* \OO_{{Z}}$ with $H^m(Y,\Sym((\S^{\oplus r})^*))$,
  which is zero for $m \geq 1$. Now use the Cauchy formula
  \cite[Theorem 2.3.2]{weyman} to write
  \[
    \Sym((\S^{\oplus r})^*) = \bigoplus_\lambda \SS^\lambda(k^r) \otimes \SS^\lambda(\S^*),
  \]
  the sum over partitions $\lambda$ with at most $r$ parts. The result
  follows from the additivity of global sections over direct sums and
  the linear independence of the Schur functors.
\end{proof}

\begin{remark}
  Weyman's geometric method can be
  used to construct a free resolution of $k[X_v]$. Define
  \[
    F_i = \bigoplus_{j \geq 0} H^j( Y, \bigwedge^{i+j} (\mathcal{Q}^*)^{\oplus r}) \otimes_k A(-i-j),
  \]
  where $\mathcal{Q} = k^n/\mathcal{S}$ is the tautological quotient
  bundle on the Grassmannian. \cite[Theorem~5.1.3]{weyman} asserts
  that there are minimal differentials $d_i : F_i \to F_{i-1}$ so that
  $F_\bullet$ is a finite free resolution of $k[X_v]$ as an
  $A$-module. By the dual Cauchy theorem, the cohomologies of
  $\bigwedge^m (\mathcal{Q}^*)^{\oplus r}$ can be expressed in terms
  of Schur functors applied to $\mathcal{Q}^*$. While it can be shown
  that $H^i(Y,\mathbf{S}^\lambda(\mathcal{Q})) =0$ for $i>0$, similar
  to the argument above, the higher cohomology groups of
  $\mathbf{S}^\lambda(\mathcal{Q}^*)$ (and dually, $\mathbf{S}^\lambda(\mathcal{S})$) are currently unknown. This is
  the subject of future work.
\end{remark}
\section{Matroid invariance of the equivariant $K$ class of $X_v$}\label{sec:matroidinv1}
To prove the first part of the theorem in the introduction, the
matroid invariance of the class of $X_v$ for \textit{any} matrix $v$,
we consider the case when $v$ has rank $r \leq n$. By
\cite[Proposition 6.6]{moc}, it is enough to prove the theorem in this
special case. We thus assume without loss of generality that $v$ has
rank $r$ for the remainder of our article.

In this section we give a formula for the equivariant $K$ class of
$X_v$ in terms of the class of $Y$, the $T$ orbit closure in $Gr(r,n)$
through the row space of $v$. That this formula solves the problem of the matroid invariance of the class of $X_v$ follows from a result of Speyer.
\begin{theorem}[Speyer {\cite[proof of Prop.~12.5]{speyer}}]\label{thm:speyer}
  The class of a torus orbit closure $Y$ in $K^T_0(Gr(r,n))$ depends only of the matroid of a point in the big orbit of $Y$.
\end{theorem}

As mentioned above, the class of $Y$ can be expressed as a
linear combination of classes of Schubert varieties.
These have lifts to~$\AA^{r \times n}$,
known as \emph{matrix Schubert varieties} \cite{km},
which we will use in an analogous expression for the class of~$X_v$.
Let $\Omega_\lambda$ be a Schubert variety in $Gr(r,n)$, and let
$X_\lambda^\circ$ denote the locus of $r$-by-$n$ matrices in
$\AA^{r \times n}$ whose row space lies in $\Omega_\lambda$. The
closure of $X_\lambda^\circ$ in $\AA^{r \times n}$, denoted
$X_\lambda$, is a matrix Schubert variety.

\begin{theorem}\label{thm:matroidinv}  Let $v$ be a rank $r$ matrix.
  \begin{enumerate}
  \item  The class of the structure sheaf of
    $X_v$ in $K^G_0(\AA^{r \times n})$ can be determined from the
    matroid of $v$ alone.
  \item The class of $X_v$ can be written as a
    $\ZZ[t_1^{\pm 1},\dots,t_n^{\pm 1}]$-linear combination of the
    classes of matrix Schubert varieties, $X_\lambda$, where
    $\lambda_1 \leq n-r$.
  \end{enumerate}
\end{theorem}
\begin{proof}[Proof of \cref{thm:matroidinv}(1)]
  Let $Y \subset Gr(r,n)$ be the torus orbit closure through the row
  span of $v$. Recall that $G = \GL_r \times T$. Then $G$ acts on the
  Grassmannian $Gr(r,n)$, where the $\GL_r$-factor acts
  trivially. So there is a natural isomorphism of
  $R(G) = R(\GL_r)\otimes R(T)$-modules,
  \[
    R(\GL_r) \otimes K^T_0(Gr(r,n)) \approx K^G_0(Gr(r,n)).
  \]
  Under this isomorphism $1 \otimes [\OO_Y]$ maps to $[\OO_Y]$.

  Now, we use \cite[Corollary 12]{handbook} to see that the projection
  map $\pi : \AA^{r \times n} \times Gr(r,n) \to Gr(r,n)$ induces a
  pullback isomorphism
  \[
    \pi^*:K^G_0(Gr(r,n)) \to K^G_0(\AA^{r \times n} \times Gr(r,n)).
  \]
  The pullback of $\OO_Y$ is $\OO_{\AA^{r \times n} \times Y}$. We
  now multiply by the class of the vector bundle $\S^{\oplus r}$,
  which is say, the class of its locally free sheaf of sections. The
  result is the class of $\OO_{Z}$, which is the class of the sheaf of
  sections of the restriction of $\S^{\oplus r}$ to $Y$.  Now we apply
  $s_*$ to $[\OO_{{Z}}]$, which by \cref{prop:ratsingGrassmannian},
  gives $[\OO_{X_v}]$. In summary,
  \[
    [\OO_{X_v} ] = s_*( [\mathcal{S}^{\oplus r}]\cdot \pi^* [\OO_Y]),
  \]
  and since the right side is determined by the matroid of $v$, by \cref{thm:speyer}, so is the left.
\end{proof}
\begin{proof}[Proof of \cref{thm:matroidinv}(2)]
  In $K^G_0(Gr(r,n))$ write
  $[\OO_Y] = \sum_\lambda c_\lambda [\Omega_\lambda]$, where the sum
  is over partitions $\lambda$ with $\lambda_1 \leq n-r$ and
  $c_\lambda \in \ZZ[t_1^{\pm 1},\dots,t_n^{\pm 1}]$ (since $\GL_r$
  acts trivially on $Y$). A matrix Schubert variety $X_\lambda$ has
  rational singularities \cite[Theorem 2.4.3]{km} and is
  (partially) resolved by the total space of the vector bundle
  $\mathcal{S}^{\oplus r}$ restricted to $\Omega_\lambda$, 
  just as we have done for $X_v$ in part~(1).
  Thus,
  \begin{align*}
    [\OO_{X_v} ] &= s_*( [\mathcal{S}^{\oplus r}]\cdot \pi^* [\OO_Y])\\
                 &=
    s_*\left([\mathcal{S}^{\oplus r}]\cdot \pi^* \sum_\lambda c_\lambda
                   [\Omega_\lambda]\right) \\
                 &=  \sum_\lambda c_\lambda
                   s_*\left([\mathcal{S}^{\oplus r}]\cdot \pi^*[\Omega_\lambda]\right) = \sum_\lambda c_\lambda [\OO_{X_\lambda}]. \qedhere
  \end{align*}
\end{proof}

We note the following immediate corollary of the proof.
\begin{corollary}\label{cor:sameschubertexp}
  In the notation of the previous results, if we uniquely expand
  $[\mathcal{O}_Y]$ in the Schubert basis,
  \[
    [\mathcal{O}_{Y}] = \sum_\lambda c_\lambda [\Omega_\lambda] \in
    K^T_0(Gr(r,n)),
  \]
  then, for the same coefficients $\{c_\lambda\}$ we have,
  \[
    [\mathcal{O}_{X_v}] = \sum_\lambda c_\lambda [X_\lambda] \in
    K^G_0(\AA^{r \times n}).
  \]
\end{corollary}
\section{Consequences of matroid invariance}\label{sec:consequences}
In this section we state the consequences of the matroid invariance of
$[X_v]$, which we studied further in \cite{moc}. Let $v_1,\dots,v_n$
denote the columns of the matrix $v$. Write $G(v)$ for the span in
$(k^r)^{ \otimes n}$ of the tensors
\[
(gv_1) \otimes (g v_2) \otimes \dots \otimes (gv_n), \quad g \in \GL_r.
\]
Clearly $G(v)$ is a representation of $\GL_r$.
\begin{corollary}
  The class of $G(v)$ in $R(\GL_r)$ is determined by the matroid of~$v$.
\end{corollary}

Let $\mathfrak{S}(v)$ denote the span in $(k^r)^{ \otimes n}$ of the
tensors
\[
  v_{w_1} \otimes v_{w_2} \otimes \dots \otimes v_{w_n}, \quad w \in S_n
\]
Clearly $\mathfrak{S}(v)$ is a representation of $S_n$.
\begin{corollary}
  The class of $\mathfrak{S}(v)$ in $R(S_n)$ is determined by the
  matroid of~$v$.
\end{corollary}
Describing the irreducible decomposition of these representations was
the motivation for studying $X_v$ and its class in $K$-theory. The
first corollary follows because the character of $G(v)$ is the
coefficient of $t_1t_2 \dots t_n$ in $\Hilb(k[X_v])$. The second
corollary follows from the first by Schur-Weyl duality. For the proofs
of these statements and further information, we refer the reader to
\cite{moc}.
\section{Explicit formula for the $K$-class of $X_v$}\label{sec:Kpoly}
In this section we compute a formula for the class of $X_v$ in
$K^G_0(\AA^{r\times n})$ using equivariant localization. This class is
represented as the numerator of the Hilbert series $\Hilb(k[X_v])$, as
explained in \cref{sec:K}.

Let $\mathcal{E}$ be a finite dimensional $T$-equivariant vector
bundle over $Y$. 
We use the same symbol $\mathcal{E}$ for its sheaf of sections.
The equivariant Euler characteristic of $\mathcal{E}$
is $\chi_T(\mathcal{E}) = \sum_i (-1)^i [H^i(Y,\mathcal{E})]$, where
$[-]$ means to compute the character as a representation of $T$.

Since we have shown the $G$-equivariant identification
$H^0( Y,\Sym((\mathcal{S}^*)^{\oplus r})) = k[X_v]$, and the higher cohomology
groups of $\Sym((\mathcal{S}^*)^{\oplus r})$ on~$Y$ vanish, we can compute the
Hilbert series of $k[X_v]$ by computing the $T$-equivariant Euler
characteristic of $\Sym((\mathcal{S}^*)^{\oplus r})$. 

To
state the formula for $\chi_T(\mathcal{E})$ we let $T_B Y$ be the
Zariski tangent space of $Y$ at a fixed point $B$, and let
$C_B \subset T_B Y$ be the tangent cone. The localization formula given in the corollary to 
\cite[Theorem~1.5]{anderson} states that there is an equality in
$\QQ(u_1,\dots,u_r,t_1,\ldots,t_n)$, the field of fractions of $R(T)$,
to wit
\[
  \chi_T(\mathcal{E}) = \sum_{B \in Y^T} [\mathcal{E}_B] \cdot \Hilb(C_B),
\]
where the sum is over $T$-fixed points $B$ of~$Y$. This formula finds
its motivation in the Atiyah-Bott-Berline-Vergne localization formula
for smooth varieties. (Alternately, one could pass to a
desingularization of $Y$ and compute the Euler characteristic of the
pullback of $\Sym((\mathcal{S}^*)^{\oplus r})$ using the 
aforementioned formula.) Above,
$\Hilb(C_B)\in\QQ(u_1,\dots,u_r,t_1,\ldots,t_n)$ is the Hilbert series
of $C_B \subset T_B$ and $\mathcal{E}_B$ is the fiber of $\mathcal{E}$
over $B$.  The formula applies in our case since $Y$ has finitely many
$T$-fixed points and the trivial character does not appear in
$[T_B Y]$. The Hilbert series $\Hilb(C_B)$ are referred to as
``equivariant multiplicities'' in \cite{anderson} (see Proposition~6.3
in \textit{loc.\,cit.}).

We now compute $\chi_T(\Sym^m((\S^*)^{\oplus r}))$ using this
formula. The $T$-fixed points of $Y$ correspond to the bases $B$ of
the matroid $M$ of $v$. Over a fixed point $B$, we can compute the
character of $\Sym((\S^*)^{\oplus r})$ to be
\[
  \prod_{j \in B} \prod_{i \in [r]} \frac{1}{1-u_it_j}
\]
The equivariant multiplicities were computed in \cite[Lemma 5.2]{chow} as
\begin{equation}\label{eq:em}
  \Hilb(C_B) = \sum_{(w_1,\dots,w_n)} \prod_{i=1}^{n-1} \frac{1}{1-t_{w_{i+1}}/t_{w_i}},
\end{equation}
where the sum is over those permutations $w \in S_n$ whose
lexicographically first basis is $B$.  Putting all this together and
condensing the summation over fixed points and permutations gives
\[
  \Hilb(k[X_v]) =
  \chi_T(\Sym((\S^*)^{\oplus r}) = 
  \sum_{w \in S_n}   \prod_{j \in B(w)} \prod_{i \in [r]} \frac{1}{1-u_it_j}
  \prod_{i=1}^{n-1} \frac{1}{1-t_{w_{i+1}}/t_{w_i}},
\]
where $B(w)$ denotes the lexicographically first basis of the matroid
$M$ occurring in the list $w = (w_1,\dots,w_n)$. Finally, to obtain
the $K$-theory class $\K(k[X_v])$ we multiply by
$\prod_{j \in [n]} \prod_{i \in [r]} (1-u_it_j)$.
\begin{theorem}\label{thm:anderson}
  The class of $X_v$ in $K^G_0(\AA^{r \times n})$ is given by the
  formula
  \[
    \K(k[X_v]) =   \sum_{w \in S_n}   \prod_{j \notin B(w)} \prod_{i \in [r]} ({1-u_it_j}) \cdot 
    \prod_{i=1}^{n-1} \frac{1}{1-t_{w_{i+1}}/t_{w_i}}.
  \]
\end{theorem}
Note that the left side is \emph{a priori} a Laurent polynomial in
$u_1,\dots,u_r$ and $t_1,\dots,t_n$. We have thus shown the
claims of the second part of the theorem from the introduction.
\begin{example}
Let $v$ represent the matroid $M$ of rank $r=2$ on $n=4$ elements with bases
\[\{\{1,2\},\{1,3\},\{2,3\},\{1,4\},\{2,4\}\}.\]
We first evaluate the equivariant multiplicities. 
One way to do this is afforded by the proof of \cite[Lemma 5.2]{chow}:
they are Hilbert series of affine toric varieties, and therefore lattice point enumerators of cones.
The cones in question are the tangent cones to $P(M)$.
For this $M$, the toric variety $Y$ is smooth except at the fixed point $\operatorname{span}\{e_1,e_2\}$,
so the tangent cones other than $C_{\{1,2\}}$ are unimodular simplicial cones. We get
\[\Hilb(C_{\{1,3\}}) = \frac1{(1-t_2/t_3)(1-t_4/t_3)(1-t_2/t_1)}\]
and
\[\Hilb(C_{\{1,2\}}) = \frac{1 - t_3t_4/t_1t_2}{(1-t_3/t_2)(1-t_4/t_2)(1-t_3/t_1)(1-t_4/t_1)},\]
and the other three equivariant multiplicities are images of~$\Hilb(C_{\{1,3\}})$ under
transposing 1 and~2 and/or transposing 3 and~4.

The reader may check that the same rational functions are obtained from \eqref{eq:em}.
To facilitate this we describe the permutations $w$ that achieve $B=B(w)$ for each basis~$B$.
If $\{w_1,w_2\}$ is a basis of~$M$, then $B(w) = \{w_1,w_2\}$; this gives four permutations for each basis.
The $4!-5\cdot4 = 4$ permutations unaccounted for are $3412$, $3421$, $4312$, and $4321$, 
whose lexicographically first bases are $\{1,3\}$, $\{2,3\}$, $\{1,4\}$, and $\{2,4\}$ respectively.

Using these equivariant multiplicities, the sum in \Cref{thm:anderson} works out to
\begin{equation}\label{eq:example K}
\K(k[X_v]) = 1 - u_1u_2t_3t_4.
\end{equation}
Elements $3$ and~$4$ are parallel, so $M$ is a parallel extension of
the uniform matroid $U_{2,3}$, and therefore the above class can also
be computed from Theorem~9.3 or Proposition~9.5 of~\cite{moc}. The
procedure described in \cite[Theorem~9.3]{moc} is to apply a Demazure
divided difference operator $\delta_3$ to the class of the matroid
$U_{2,3} \oplus U_{0,1}$, which is $(1-u_1t_4)(1-u_2t_4)$ since the
associated matrix orbit closure is a linear subspace of
$\AA^{2 \times 4}$. Application of the divided difference gives
\begin{align*}
  \delta_3 ( (1-u_1t_4)(1-u_2t_4) )  &= \frac{(1-u_1t_4)(1-u_2t_4) - t_4/t_3 (1-u_1t_3)(1-u_2t_3)}{1 - t_4/t_3}\\&= 1-u_1u_2t_3t_4.
\end{align*}

We check the agreement of this class with \Cref{thm:binglin},
first deriving the Chow class featuring in that theorem from \eqref{eq:example K}.
The class of $X_v$ in $K^{T}_0(\AA^{r\times n})$, forgetting the $\GL_r$-action, 
can be computed by evaluating the $u_i$ at~1.
The resulting class is $1 - t_3t_4$.
The class of $V = p(X_v)$ in 
\[K_0((\PP^{r-1})^n) = \ZZ[t_1,\dots,t_n]/((1-t_1)^r,\dots, (1-t_n)^r)\]
% We need the class of (column i is zero) subvariety to map to zero in
% the quotient. This K class is (1-t_i)^r. Thus, at worst we have a
% surjection from this quotient to the real K theory. However, they
% have the same rank as free Z modules so they must be isomorphic.
is represented by this same polynomial.
Finally, \cite[Section~8.5]{millerSturmfels} asserts that the class of $V$ in the Chow ring $A^*((\PP^{r-1})^n)$
is the sum of terms of lowest degree after substituting $1-t_i$ for $t_i$ ($i\in[n]$) in the last $K$-class.
The substitution yields
\[1 - (1-t_3)(1-t_4) = t_3 + t_4 - t_3t_4\]
so the Chow class of $V$ is $t_3+t_4$.

For the other side of the comparison we must compute the set $S(M)$.
As singletons have rank~1 in~$M$ it follows that $S(M)\subseteq\{0,1\}^4$,
and of the remaining inequalities defining $S(M)$ the only one not automatically satisfied once 
$\sum_{i=1}^4 s_i = 2^2-1 = 3$
is $s_3+s_4<2$.  Therefore
\[S(M) = \{(1,1,0,1),(1,1,1,0)\},\]
from which \Cref{thm:binglin} also produces the Chow class $t_3+t_4$.
\end{example}

We now address the polynomiality of the rational function
\[
  \K(M) = \sum_{w \in S_n} \prod_{j \notin B(w)} \prod_{i \in [r]}
  ({1-u_it_j}) \cdot \prod_{i=1}^{n-1}
  \frac{1}{1-t_{w_{i+1}}/t_{w_i}}.
\]
for arbitrary matroids $M$, which we will call the $K$-class of $M$.
\begin{theorem}\label{thm:valuative}
  For any matroid $M$ of rank $r$ on $[n]$, $\K(M)$ is a polynomial in
  $u_1,\dots,u_r,t_1,\dots,t_n$.
\end{theorem}
\begin{proof}
  For a fixed permutation $w \in S_n$ let $B \mapsto [B =B(w)]$ be the
  indicator function of $B(w)$, and write $D(w)$ for $\prod_{i=1}^{n-1} \frac{1}{1-t_{w_{i+1}}/t_{w_i}}$. Then,
  \begin{align*}
    \K(M) &= \sum_{w \in S_n} \prod_{j \notin B(w)}\prod_{i \in [r]} (1-u_it_j) \cdot D(w)\\
          &= \sum_{B \in \binom{[n]}{r}}\sum_{w \in S_n} \prod_{j \notin B}\prod_{i \in [r]} (1-u_it_j) \cdot D(w) \cdot [B = B(w)].  
  \end{align*}
  %For each $w \in S_n$, define the function $f_w$ from the set of matroids to
  %the ring of rational functions
  %\[
  %  f_w(M)= \sum_{B \in \binom{[n]}{r}}  [B = B(w)] \cdot \prod_{j \notin B}\prod_{i \in [r]} (1-t_j/u_i) .
  %\]
  If $B=\{w_{i_1},\ldots,w_{i_r}\}$ with $i_1<\cdots<i_r$,
  then $B=B(w)$ if and only if
  $\operatorname{rk}_M(\{w_1,\ldots,w_i\}) = \max\{j:i_j\leq i\}$ for all $1\leq i\leq n$.
  Therefore, viewed as a function of~$M$,
  the function $[B =B(w)]$ is valuative on matroid polytope subdivisions by \cite[Proposition 5.3]{df}.
  In the expansion of $\K(M)$ above the only factors depending on~$M$ are the expressions $[B =B(w)]$,
  so $\K(M)$ is a $\QQ(u_1,\dots,u_r,t_1,\dots,t_n)$-linear combination
  of valuative functions, and is therefore valuative itself.
  
  The dual to the abelian group of valuative functions is spanned by Schubert matroids for all orderings of the ground set $[n]$
  \cite[Theorem 5.4]{df} (see also the discussion after that work's Theorem~6.3).
  That is, if $M$ is an arbitrary matroid, then there exist
  Schubert matroids $M_1,\ldots,M_k$ and integers $a_1,\ldots,a_k$ such that
  $\K(M) = \sum_{i=1}^k a_i\K(M_i)$.
  All Schubert matroids are representable over any infinite field,
  so each $\K(M_i)$ is a polynomial because it is the $K$-theory class of a sheaf.
  Thus $\K(M)$ is a polynomial too.
\end{proof}
\begin{remark}
  The polynomiality of the formula for $\K(k[X_v])$ can be deduced
  directly from the cohomological variant of this result proved in
  \cite[Theorem~1]{spinkTseng} (Hunter Spink, private
  communication). From this one can conclude that the right side above
  defines an honest element of $K_0^G(\AA^{r \times n})$ (as opposed
  to an element of its field of fractions) and further, using
  equivariant localization, that it is a lift of
  $[\mathcal{O}_Y] \in K_0^T(Gr(r,n))$. However, without rational
  singularities of $X_v$, or some other means, one cannot conclude
  that this natural lift of $[\mathcal{O}_Y]$ is the class of $X_v$.
\end{remark}

\section{$K$-theoretic positivity}
In this section we consider the Schubert expansion of the class of
matrix orbit closures in $\AA^{r \times n}$. By
\cref{cor:sameschubertexp}, this is equivalent to the Schubert
expansion of a torus orbit closure $Y$ in the Grassmannian.

For matroids realizable over $\CC$, we deduce $K$-theoretic positivity
of the class of $X_v$ via results of Anderson, Griffeth and Miller
\cite{agm}. We conjecture a strengthening of this property for all
matroids and pose a series of progressively weaker conjectures
concerning different types of positivity.

\subsection{Positivity for  matroids realizable over $\CC$}
Let $M$ be a matroid realized be a complex $r$-by-$n$ matrix
$v$. Write $\K(M) = \K(k[X_v])$ for the polynoimal of \cref{sec:Kpoly}.

There are two types of positivity that $\K(k[X_v])$ could exhibit,
being a simultaneous lift of the class of the $T^n$-equivariant
subvariety $Y \subset Gr(r,n)$ as well as the class of the
$T^r$-equivariant subvariety $p(X_v) \subset (\PP^{r-1})^n$. Our
primary concern will be the positivity obtained from the Grassmannian,
although in this subsection we will consider both types of positivity.

In order to state the first result we need to be explicit about which
(matrix) Schubert varieties we use to express our classes. For a
partition $\lambda = (\lambda_1 \geq \dots \geq \lambda_r)$,
$X_\lambda^\circ$ is the set of matrices $m \in \AA^{r \times n}$
whose rank increases occur in columns
$n-r+1 - \lambda_1  < n-r+2-\lambda_{r-1} < \dots < n-\lambda_r$, when the
matrix is read columnwise left-to-right. The matrix Schubert variety
is $X_\lambda =\overline{ X_\lambda^\circ}$. Thus $X_\lambda$ is a
point when $\lambda = (0,\dots,0)$ and
$X_\lambda = \AA^{r \times n}$ when $\lambda = (n-r,\dots,n-r)$. In this
way $X_\lambda$ has codimension
$|\lambda| = \lambda_1 + \dots + \lambda_r$. The class of $X_\lambda$
in $K^G_0(\AA^{r \times n})$ is given by the double Grothendieck
polynomial $\mathfrak{S}_\lambda(u,t)$ \cite[Theorem~A]{km}; explicit
formulas for these polynomials can be found in \cite{km,kmy}. 
\begin{proposition}\label{prop:pos1}
  Let $v \in \AA^{r \times n}$ be a rank $r$ matrix realizing a
  matroid $M$ with $e$ connected components. Write $\K(M)$ in terms of
  double Grothendieck polynomials,
  \[
    \K(M) = \sum_{\lambda: \lambda_1 \leq n-r} c_\lambda(t)
     \mathfrak{S}_\lambda(u,t).
  \]
  Then, the Laurent polynomials
  $(-1)^{r(n-r) - (n-e) - |\lambda|} c_\lambda(t) \in \ZZ[t_1^{\pm
    1},\dots,t_n^{\pm 1}]$ are polynomials in
  $t_2/t_1 - 1,t_3/t_4 - 1,\dots,t_n/t_{n-1} - 1$ with positive
  integer coefficients.
\end{proposition}
% codim(X_v) = r(n-r) - n + e
Recall that a matroid $M$ is connected if it is indecomposable with
respect to direct sum (simplicial complex join). Every matroid can be
expressed uniquely as a direct sum of connected matroids. When $M$ is
realized by a matrix $v \in \AA^{r \times n}$, then $M$ is a direct
sum of $e$ connected matroids if and only if the codimension of $X_v$
is equal to $r(n - r) - (n -e)$. The effect of direct sum on
$K$-classes is investigated in \cite[Section 7]{moc}.
\begin{proof}
  By \cref{cor:sameschubertexp}, it suffices to prove the
  class of the toris orbit closure $Y \subset Gr(r,n)$ through
  $\pi(v)$ has the analogous property in $K^{T}_0(Gr(r,n))$. This
  follows from \cite[Corollary~5.1]{agm} since $Y$ has rational
  singularities and is of codimension $r(n-r) -(n-e)$.
\end{proof}

\begin{example}\label{ex:u24}
  Let $v \in \AA^{2 \times 4}$ be a generic matrix, whose matroid is
  therefore the uniform matroid $U_{2,4}$. We have 
  \[
    \K(k[X_v]) = 1- u_1^{2}u_2^{2} t_1t_2t_3t_4,
  \]
  which can be seen from either \cref{thm:anderson},
  \cite[Proposition~5.2]{moc} or the fact that $X_v$ is a hypersurface
  in $\AA^{2 \times 4}$. We expand the above polynomial in terms
  of the classes of the matrix Schubert varieties and obtain
  \begin{multline*}
    1- u_1^{2}u_2^{2} t_1t_2t_3t_4 = t_1^{-1}t_4 [\OO_{X_{(2,1)}}] - t_1^{-1}t_4 [\OO_{X_{(2,0)}}] - t_1^{-1}t_4
    [\OO_{X_{(1,1)}}] +\\ (t_1^{-1}t_4 + t_1^{-1}t_2^{-1}t_3t_4) [\OO_{X_{(1,0)}}]
    - (t_1^{-1}t_2^{-1}t_3t_4 - 1)[\OO_{X_{(0,0)}}].
  \end{multline*}
  Writing $\beta_i = t_{i+1}/t_i-1$, $1 \leq i \leq 3$, we see that
  \[
    t_1^{-1}t_4 = (\beta_1 + 1)(\beta_2 + 1)(\beta_3 + 1)
  \] and             
  \[
    t_1^{-1}t_2^{-1}t_3t_4 = (\beta_1 + 1)(\beta_2 + 1)^2(\beta_3 + 1)
  \]
  have positive coefficients in the $\beta_i$,
  and the right hand side of the latter has 1 as a term so that $t_1^{-1}t_2^{-1}t_3t_4 - 1$ is positive in the $\beta_i$ as well.
  These computations verify the result of \cref{prop:pos1}.
\end{example}

% \begin{proposition}\label{prop:pos2}
%   Let $v \in \AA^{r \times n}$ have no columns equal to zero. We may
%   write
%   \[
%     \K(k[X_v]) = \sum_\alpha d_\alpha \prod_{j=1}^n\prod_{i=1}^{\alpha_j} (1-t_j/u_i),
%   \]
%   where the sum ranges over compositions
%   $\alpha = (\alpha_1,\dots,\alpha_n)$ with $0 \leq \alpha_i \leq r-1$, and the coefficients $d_\alpha$
%   are Laurent polynomials in $u_1,\dots,u_r$. Then,
%   $(-1)^{\codim(X_v) - \sum_i \alpha_i} d_\alpha$ is a polynomial in
%   $u_1/u_2 - 1, u_2/u_3-1,\dots,u_{r-1}/u_{r}-1$ with positive integer
%   coefficients.
% \end{proposition}

\begin{proposition}\label{prop:pos2}
  Let $v \in \AA^{r \times n}$ be a realization of a loopless matroid $M$ having $e$ connected components. We may
  write
  \[
    \K(M) = \sum_\alpha d_\alpha \prod_{j=1}^n\prod_{i=1}^{\alpha_j} (1-u_it_j),
  \]
  where the sum ranges over compositions
  $\alpha = (\alpha_1,\dots,\alpha_n)$ with $0 \leq \alpha_i \leq r-1$, and the coefficients $d_\alpha$
  are Laurent polynomials in $u_1,\dots,u_r$. Then,
  $(-1)^{r(n-r) - (n -e) - \sum_i \alpha_i} d_\alpha$ is a polynomial in
  $u_2/u_1 - 1, u_3/u_2-1,\dots,u_{r}/u_{r-1}-1$ with positive integer
  coefficients.
\end{proposition}

\begin{proof}
  As discussed in \cref{sec:ratsing}, $(\PP^{r-1})^n$ can be regarded
  as a flag variety. Hence, $K^{T^r}_0((\PP^{r-1})^n)$ has a
  $\ZZ[u_1^{\pm 1},\dots,u_r^{\pm 1}]$-module basis given by the
  classes of Schubert varieties, which now are products of linear
  spaces indexed by compositions $\alpha$ of length $n$ whose parts
  $\alpha_i$ satisfy $0 \leq \alpha_i \leq r-1$. The Schubert variety
  $\Sigma_\alpha$ of $\alpha$ consists of those points
  $(p_1,\dots,p_n)$ where $(p_i)_j = 0$ for $0 \leq j \leq \alpha_i$,
  and has codimension $\sum_i \alpha_i$. The corresponding affine
  analogue $X_\alpha$ of $\Sigma_\alpha$ clearly has rational
  singularities since it is a linear subspace of $\AA^{r \times
    n}$. The $K$-class of $X_\alpha \subset \AA^{r \times n}$ is
  readily computed to be
  $\prod_{j=1}^n\prod_{i=1}^{\alpha_j} (1-u_it_j)$. When we work in
  $K^{T^r}_0((\PP^{r-1})^n)$, we have
  \[
    [\OO_{p(X_v)}] = \sum_{\alpha} d_\alpha [\OO_{\Sigma_\alpha}]
  \]
  to which we apply \cite[Corollary~5.1]{agm}. Since the codimension of $p(X_v)$ is $r(n-r) - (n-e)$, we conclude that
  $(-1)^{r(n-r) - (n -e) - \sum_i \alpha_i}d_\alpha$ is a polynomial in
  $u_2/u_1-1$, $u_3/u_2-1$, \dots $u_{r}/u_{r-1}-1$ with positive
  integer coefficients.  Our goal is to lift this result to
  $\K(k[X_v])$, following the proof of \cref{thm:matroidinv}.

  Let $q_i$ be the projection of
  $\AA^{r \times n} \times (\PP^{r-1})^n$ to its $i$th factor, and let
  $\mathcal{E} \subset \AA^{r \times n}$ be the vector bundle of
  \cref{sec:ratsing}. As before if $X \subset \AA^{r \times n}$ has
  rational singularities (and does not lay in the space of matrices
  with a zero column) then
  ${q_1}_* \left([\mathcal{E}] \cdot q_2^* [\OO_{p(X)}]\right) = [\OO_X]$. Then, as all the
  varieties below have rational singularities,
  \begin{multline*}
    [\OO_{X_v}] = {q_1}_* \left([\mathcal{E}] \cdot q_2^*[\OO_{p(X_v)}]\right) 
                = {q_1}_* \left([\mathcal{E}] \cdot q_2^* \sum_\alpha  d_\alpha [\OO_{\Sigma_\alpha}] \right)\\
                =  \sum_\alpha  d_\alpha {q_1}_* \left([\mathcal{E}]\cdot q_2^*  [\OO_{\Sigma_\alpha}]\right)=  \sum_\alpha  d_\alpha [\OO_{X_\alpha}].
  \end{multline*}
  The result follows.
\end{proof}
\begin{example}
  We continue the above example with a generic $2$-by-$4$ matrix $v$
  and expand
  \[
    \K(k[X_v]) = 1- u_1^{2}u_2^{2} t_1t_2t_3t_4,
  \]
  as described in \cref{prop:pos2}. We have
  \[
    1- u_1^{2}u_2^{2} t_1t_2t_3t_4 = 1 - (u_2/u_1)^2 \sum_{\alpha}
    (-1)^{|\alpha|-1} \prod_{j : \alpha_j > 0} (1-u_1t_j)
  \]
  where the sum is over compositions $\alpha$, with $\alpha_i \leq 1$ for all $i$. 
  The coefficient of $1 = [\OO_{X_{(0,0,0,0)}}]$ on the right hand side is
  $-((u_2/u_1)^2 - 1)$, 
  and the remaining coefficients are greater by unity, $\pm(u_1/u_2)^2$.
  Since 
  \[(u_2/u_1)^2 - 1 = (u_2/u_1 - 1)^2 + 2 (u_2/u_1 - 1),\]
  we have verified the proposition.
\end{example}

\subsection{Positivity conjectures for all matroids}
We begin by conjecturing a strengthening of \cref{prop:pos1} for all
matroids.
\begin{conjecture}\label{conj:sqfree} Let $M$ be a  matroid  of rank $r$ on $n$ elements
  with $e$ connected components. Write $\K(M)$ in terms of double Grothendieck
  polynomials,
  \[
    \K(M) = \sum_{\lambda: \lambda_1 \leq n-r} c_\lambda(t)
     \mathfrak{S}_\lambda(u,t).
  \]
  Then, the Laurent polynomials
  $(-1)^{r(n-r) - (n-e) - |\lambda|} c_\lambda(t) \in \ZZ[t_1^{\pm
    1},\dots,t_n^{\pm 1}]$ can be expressed
  (non-uniquely) as a square-free polynomial in the quantities
  $t_j/t_i - 1$, $1 \leq i < j \leq n$, with positive integer
  coefficients.
\end{conjecture}
This property has been observed to hold for $K$-theoretic
Littlewood-Richardson coefficients \cite[Corollary~1.5]{py}, although
no geometric reason is available at this time.
\begin{example}
  We continue \cref{ex:u24}, where $M = U_{2,4}$. It suffices to see
  that $t_1^{-1}t_2^{-1}t_3 t_4$ can be a written as a square-free
  polynomial in $t_j/t_i - 1$, $1 \leq i < j \leq 4$, with
  non-negative integer coefficients, since all other coefficients
  evidently already have the desired property. We have, for example,
  \[
    t_1^{-1}t_2^{-1}t_3 t_4 = (t_4/t_2 - 1)(t_3/t_1 - 1) + (t_3/t_1 -
    1) + (t_4/t_2 - 1) + 1,
  \]
  so the conjecture holds in this case.
\end{example}
Of course, \cref{conj:sqfree} implies the following slightly weaker
version, with the same conclusion as~\cref{prop:pos1}.
\begin{conjecture} Let $M$ be a  matroid  of rank $r$ on $n$ elements
  with $e$ connected components. Write $\K(M)$ in terms of double Grothendieck
  polynomials,
  \[
    \K(M) = \sum_{\lambda: \lambda_1 \leq n-r} c_\lambda(t)
     \mathfrak{S}_\lambda(u,t).
  \]
  Then, the Laurent polynomials
  $(-1)^{r(n-r) - (n-e) - |\lambda|} c_\lambda(t) \in \ZZ[t_1^{\pm
    1},\dots,t_n^{\pm 1}]$ can be expressed as a polynomial in the
  quantities $t_{i+1}/t_i - 1$, $1 \leq i< n$, with positive integer
  coefficients.
\end{conjecture}
The existence of the decomposition in the conjecture
follows from the argument in \cref{thm:valuative}. 
Indeed, $\K(M)$ is a valuative matroid invariant
and hence it suffices to prove the existence of the decomposition for
matroids realizable over $\CC$, which we did in \cref{prop:pos1}. In
general, when writing $\K(M)$ as a $\ZZ[t_1,\dots,t_n]$-linear
combination of $K$-classes of matroids realizable over $\CC$ one loses
control on the signs of coefficients involved, so valuativity does not
immediately reduce the conjecture to the realizable case.

We now state the non-equivariant analogue of these results. Recall
that the matrix Schubert variety $X_\lambda \subset \AA^{r \times n}$
has its class in $K^{\GL_r}_0(\AA^{r \times n})$ equal to the
Grothendieck polynomial $\mathfrak{S}_\lambda(u)$. This polynomial
represents the class of the Schubert variety
$\Omega_\lambda \subset Gr(r,n)$ in $K_0(Gr(r,n))$. These polynomials
satisfy
$\mathfrak{S}_\lambda(u) = \mathfrak{S}_\lambda(u,(1,\dots,1))$, since
the restriction
$K^G_0(\AA^{r \times n}) \to K^{\GL_r}_0(\AA^{r \times n})$ is
obtained by setting $t_1 = \dots = t_n = 1$. Indeed, evaluating all
$t_i = 1$ in the above conjectures yields the following.
\begin{conjecture} Let $M$ be a  matroid  of rank $r$ on $n$ elements
  with $e$ connected components. Write
  $\K(M)|_{t_1 = \dots = t_n = 1}$ in terms of  Grothendieck
  polynomials,
  \[
    \K(M)|_{t_1 = \dots = t_n = 1} = \sum_{\lambda: \lambda_1 \leq n-r} c_\lambda
     \mathfrak{S}_\lambda(u).
  \]
  Then, the constants $(-1)^{r(n-r) - (n-e) - |\lambda|} c_\lambda$
  are non-negative integers.
\end{conjecture}

We now explain how these conjectures imply analogous results on
$G$-equivariant Chow classes. Let $A^\bullet_G(\AA^{r \times n})$ be
the $G$-equivariant Chow ring of $\AA^{r \times n}$. Since
$\AA^{r \times n}$ is an affine space, one may write
\begin{align*}
  A^\bullet_G(\AA^{r \times n}) =
  \ZZ[u_1,\dots,u_r,t_1,\dots,t_n]^{{S}_r}
\end{align*}
and the class of the equivariant coherent sheaf associated to an
$R$-module $N$ is its multidegree $\C(N)$ \cite[Proposition 1.9]{kms}.
There is a general procedure to extract $\C(N)$ from the
$K$-polynomial $\K(N)$ which we now describe.
\begin{proposition}\label{prop:KtoChow}
  Let $N$ be a $G$-equivariant coherent $R$-module whose class in
  $K^G_0(\AA^{r \times n})$ is represented by its $K$-polynomial
  $\K(N)$. To obtain the class of $\C(N)$ in the equivariant Chow ring
  $A^\bullet_G(\AA^{r \times n})$, replace each $u_i$ with $1-u_i$,
  each $t_j$ with $1-t_j$, expand the resulting formal power series
  and then gather the lowest non-zero term, which will be of degree
  $\codim(N)$.
\end{proposition}
\begin{proof}
  The $G$-equivariant $K$ and Chow classes of $N$ are, by definition,
  equal to the $T$-equivariant $K$ and Chow classes of $N$ where $T$
  is the maximal torus of $G$. (The $G$-equivariant $K$ and Chow rings
  are equal to the subring of Weyl group invariants in $T$-equivariant
  $K$ and Chow rings.) The result now follows from Proposition 1.9 and
  the definition of the multidegree in {\cite{kms}}. See also
  \cite[Section 8.5]{millerSturmfels}.
\end{proof}

The $G$-equivariant Chow class of $X_v$ was studied in
\cite{chow,spinkTseng}. For each matroid $M$ of rank $r$ on ground set
$\{1,2,\dots,n\}$, define a rational function by the formula
\[
  \C(M) = \sum_{w \in S_n} \prod_{j \notin B(w)} \prod_{i \in [r]}
  (u_i + t_j) \cdot \prod_{i=1}^{n-1} \frac{1}{t_{w_{i+1}}-t_{w_i}}.
\]
Since each summand in the definition of $\K(M)$ is homogeneous,
$\C(M)$ is obtained from $\K(M)$ by applying the substitutions
$u_i \mapsto 1-u_i$ and $t_j \mapsto 1-t_j$ and extracting the terms
of smallest possible degree after simplifying. By \cref{prop:KtoChow},
we have the following result. 
\begin{proposition}\label{prop:chow}
If $v \in \AA^{r \times n}$ is a realization of a rank $r$ matroid
$M$ on $n$ elements then $\C(M)$ equals the class of $X_v$ in the
$G$-equivariant Chow ring of $\AA^{r \times n}$.  
\end{proposition}
We remark that this is exactly the statement of
\cite[Theorem~1]{spinkTseng}, which obtains the result through
entirely different means.

The $G$-equivariant Chow class of the matrix Schubert variety
$X_\lambda$ is equal to the double Schur polynomial $s_\lambda(u,t)$,
which is a special case of a double Schubert polynomial for
Grassmannian permutations \cite[Theorem A]{km} (see also
\cite{kmy}). It is equal to the lowest degree term of
$\mathfrak{S}_\lambda(1-u,1-t)$ after simplification. The following
conjecture is the result of applying \cref{prop:KtoChow} to
\cref{conj:sqfree}.
\begin{conjecture}\label{conj:chow1} 
  Let $M$ be a matroid of rank $r$ on $n$ elements with $e$ connected
  components. Write $\C(M)$ in terms of double Schur polynomials,
  \[
    \C(M) = \sum_{\lambda : \lambda_1 \leq n-r} d_\lambda(t) s_\lambda(u,t).
  \]
  Then, the polynomial
  $(-1)^{r(n-r) - (n-e)-|\lambda|}d_\lambda(t) \in \ZZ[t_1,\dots,t_n]$
  can be expressed (non-uniquely) as a square-free polynomial in the
  quantities $t_i - t_j$, $1 \leq i < j \leq n$ with non-negative
  integer coefficients.
\end{conjecture}
The following weaker result holds for all matroids $M$ realizable over
$\CC$ by \cref{prop:pos1}, or by \cite[Theorem~3.2]{graham}.
\begin{conjecture}
  Let $M$ be a matroid of rank $r$ on $n$ elements with $e$ connected
  components. Write $\C(M)$ in terms of double Schur polynomials,
  \[
    \C(M) = \sum_{\lambda : \lambda_1 \leq n-r} d_\lambda(t) s_\lambda(u,t).
  \]
  Then, the polynomial
  $(-1)^{r(n-r) - (n-e)-|\lambda|}d_\lambda(t) \in \ZZ[t_1,\dots,t_n]$
  can be expressed as a  polynoimal in the
  quantities $t_{i} - t_{i+1}$, $1 \leq i < j \leq n$ with non-negative
  integer coefficients.
\end{conjecture}
We finally present the weakest but most accessible
version of our conjectures by passing from the $G$-equivariant Chow
ring of $\AA^{r\times n}$ to the $\GL_r$-equivariant Chow ring. It
follows from the above conjecture by $t_1 =\dots =t_n = 0$.
\begin{conjecture}\label{conj:chow2}
  Let $M$ be a matroid of rank $r$ on $n$ elements with $e$ connected
  components. Write $\C(M)|_{t_1 = \dots = t_n =0}$ in terms of Schur
  polynomials,
  \[
    \C(M)|_{t_1 = \dots = t_n =0} = \sum_{\lambda : \lambda_1 \leq n-r,
      |\lambda| = r(n-r)-(n-e)} d_\lambda s_\lambda(u).
  \]
  Then, the constants $d_\lambda$ are non-negative integers.
\end{conjecture}

\begin{example}
  If $M = U_{2,4}$ then $\K(M) = 1-u_1^2u_2^2t_1t_2t_3t_4$. It follows that $\C(M) = 2u_1 +2u_2 + t_1 +t_2 + t_3 +t_4$. We write
  \[
    \C(M) = 2s_{(1,0)}(u,t) - (t_1 + t_2 - t_3 -t_4) s_{(0,0)}(u,t).
  \]
  Since $-(t_1 + t_2 - t_3 - t_4) = (-1)((t_1 - t_3) + (t_2 - t_4))$, $\C(M)$
  satisfies \cref{conj:chow1}. Setting $t_1 = \dots = t_4 = 0$ we see
  at once that \cref{conj:chow2} is satisfied as well.
  % r(n-r) - n + e = 2*2 - 4 + 1 = 1.
\end{example}

\begin{remark}
  The coefficients that appear in \cref{conj:chow2} 
  are related to the Schur classes of tautological
  bundles of matroids studied in \cite{best}. When $M$ is connected
  they are degrees of such Schur classes.  
  
  Let $Y$ be the torus orbit closure of a point in $Gr(r,n)$ with matroid~$M$. 
  Write $y^T(M)$ for the class of $Y$ in $A^\bullet_T(Gr(r,n))$,
  which is represented by the multidegree $\C(M)$ in the appropriate
  quotient of $A^\bullet_G(\AA^{r \times n})$. 
  Let $y(M)$ denote the image of $y^T(M)$ in
  the ordinary (non-equivariant) Chow ring of $Gr(r,n)$. 
  Assume $M$ is connected, i.e., $Y$ has dimension $n-1$. 
  
  The permutohedral toric variety $X_{A_{n}}$ is a toric resolution of singularities of~$Y$.
  For a class $\xi\in A^\bullet(Gr(r,n))$,
  the push-pull formula for this resolution says that
  \[
    \deg_{Gr(r,n)}( y(M)\,\xi ) = \deg_{X_{A_n}}( \xi_M )
  \]
  where equivariant localization yields the following description of~$\xi_M$.
  Let $\xi^{T} \in A^\bullet_{T}(Gr(r,n))$ be an equivariant lift of~$\xi$.
  Given an $r$-subset $B$ of $[n]$, let $\xi^{T}_B \in \ZZ[t_1,\dots,t_n]$
  denote the localization of $\xi^T$ at the $T$ fixed point of
  $Gr(r,n)$ indexed by $B$. 
  For a matroid $M$ of rank $r$ on $[n]$, the collection of polynomials
  $\xi^T_M = (\xi^T_{B(w)})_{w \in S_n}$ defines a class in the
  $T$-equivariant Chow ring of~$X_{A_{n}}$,
  and $\xi_M$ is the corresponding ordinary class. 
  This is the Chow-theoretic analogue of, and is implied by,
  \cite[Proposition~3.13]{best}.
  The argument in \cite[Lemma~10.9]{best}
  extends this to non-realizable matroids. 
  % Let $f : X_{A_n} \to Y \to Gr(r,n)$ be the
  % composite of the resolution of singularities followed by the
  % inclusion. Then $\xi_M = f^* \xi$ and because
  % $\dim(X_{A_n}) = \dim(Y)$ the push-pull formula says
  % f_* (f* \xi \cdot [X_{A_n}]) = \xi \cdot [Y]. Taking degrees and
  % noting that f_* doesn't change degree does the job. 
  
  We may uniquely write
  \[
    \C(M)|_{t_1 = \dots = t_n =0} = \sum_{\lambda : \lambda_1 \leq n-r,
      |\lambda| = r(n-r)-(n-1)} d_\lambda s_\lambda(u).
  \]
  Then \cref{cor:sameschubertexp} implies that
  \[
    y(M) = \sum_{\lambda : \lambda_1 \leq n-r, |\lambda| =
      r(n-r)-(n-1)} d_\lambda [\Omega_\lambda] \in A^\bullet(Gr(r,n))
  \]
  for the same constants $d_\lambda$. For an (equivariant) vector
  bundle $\mathcal{E}$ let $s_\lambda(\mathcal{E})$ denote the Schur
  polynomial evaluated at its (equivariant) Chern roots. There is an
  equality of non-equivariant classes
  $[\Omega_{\lambda}] = s_\lambda(\mathcal{S}^\vee)$ in
  $A^\bullet(Gr(r,n))$. Let $\lambda^*$ denote the partition
  corresponding to the complement of $\lambda$ in a $r$-by-$(n-r)$
  box, so that
  $[\Omega_\lambda][\Omega_{\mu^*}] = \delta_{\lambda\mu}$. Assume
  $\mu$ is a partition of $n-1$. Pairing $y(M)$ with $[\Omega_\mu]$ and taking degree gives
  \begin{align*}
  d_{\mu^*} = \deg_{Gr(r,n)}( y(M) [\Omega_\mu]) &= \deg_{Gr(r,n)}( y(M) s_\mu(\mathcal{S}^\vee)) \\
    &= \deg_{X_{A_{n-1}}}(s_\mu(\mathcal{S}^\vee)_M)= \deg_{X_{A_{n-1}}}(s_\mu(\mathcal{S}|_M^\vee)),
  \end{align*}
  where $\mathcal{S}|_M$ is the tautological sub-bundle of the matroid
  $M$ from \cite{best}.
\end{remark}

\bibliography{ratSing}{} \bibliographystyle{alpha}
\end{document}